\documentclass[12pt, a4paper]{article}
\usepackage{amsmath, amssymb, amsthm}
\usepackage{authblk}
\usepackage{geometry}
\usepackage{setspace}
\usepackage{hyperref}
\usepackage{enumitem}
\usepackage{cite}
\usepackage{titlesec}
\titlespacing*{\section}{0pt}{3ex}{2ex}
\titlespacing*{\subsection}{0pt}{2.5ex}{1.5ex}

\newtheoremstyle{mystyle}
{3pt} 
{3pt} 
{\normalfont} 
{0pt} 
{\bfseries} 
{.} 
{5pt plus 1pt minus 1pt} 
{} 
\theoremstyle{mystyle}

\newtheorem{theorem}{Theorem}[section]
\newtheorem{lemma}[theorem]{Lemma}
\newtheorem{definition}[theorem]{Definition}
\newtheorem{remark}[theorem]{Remark}
\newtheorem{proposition}[theorem]{Proposition}

\title{On Trivial Cyclically Covering Subspaces of $\mathbb{F}_q^n$ in Non-Coprime Characteristic}

\author[1]{Shuang Li, Ph.D.}
\author[1,*]{Pingzhi Yuan}
\affil[1]{School of Mathematics, South China Normal University, Guangzhou 510631, China}
\affil[2]{School of Mathematics, South China Normal University, Guangzhou 510631, China}
\date{}

\begin{document}
	
	\maketitle
	
	\begin{abstract}
		A subspace $U$ of $\mathbb{F}_q^n$ is called \textit{cyclically covering} if the whole space $\mathbb{F}_q^n$ is the union of the cyclic shifts of $U$. The case $\mathbb{F}_q^n$ itself is the only covering subspace, is of particular interest. Recently, Huang solved this problem completely under the condition $\gcd(n, q)=1$ using primitive idempotents and trace functions, and explicitly posed the non-coprime case as an open question.
		This paper provides a complete answer to Huang's question. We prove that if $n = p^k m$ where $p = \operatorname{char}(\mathbb{F}_q)$ and $\gcd(m, p)=1$, then $h_q(p^k m) = 0$ if and only if $h_q(m) = 0$. This result fully reduces the non-coprime case to the coprime case settled by Huang. Our proof employs the structure theory of cyclic group algebras in modular characteristic.
	\end{abstract}


\noindent\rule{\textwidth}{0.4pt}

	\small
	*Corresponding Author

E-mail addresses: 2338482253@qq.com (S. Li),  mcsypz@mail.sysu.edu.cn (P. Yuan)

	\section{Introduction}
	
	\subsection{Background and Problem Statement}
	Let $\mathbb{F}_q$ be the finite field of order $q$ and let $\mathbb{F}_q^n$ be the $n$-dimensional vector space over $\mathbb{F}_q$. The \textit{cyclic shift} operator $\sigma$ on $\mathbb{F}_q^n$ is defined by
	\[
	\sigma(x_0, x_1, \ldots, x_{n-1}) = (x_{n-1}, x_0, \ldots, x_{n-2}).
	\]
	For a subspace $U \subseteq \mathbb{F}_q^n$, denote by $\sigma^i(U)$ its $i$-th cyclic shift. The subspace $U$ is called \textit{cyclically covering} if
	\[
	\mathbb{F}_q^n = \bigcup_{i=0}^{n-1} \sigma^i(U).
	\]
	Define the function
	\[
	h_q(n) = n - \min\{\dim U \mid U \text{ is a cyclically covering subspace of } \mathbb{F}_q^n\}.
	\]
	Thus $h_q(n)$ represents the largest possible codimension of a cyclically covering subspace. In particular, $h_q(n) = 0$ if and only if the only subspace covering $\mathbb{F}_q^n$ under cyclic shifts is $\mathbb{F}_q^n$ itself.
	
	The study of $h_q(n)$ was initiated by Cameron, Ellis and Raynaud \cite{CER2019}, who established its connection to Isbell's conjecture in combinatorial game theory \cite{Isbell1957, Isbell1960} and to a problem posed by Cameron \cite{Cameron1994}. They proved several foundational results, including that $h_2(n)=0$ if and only if $n$ is a power of 2, and that $h_q(k s^d)=0$ if $k \mid q-1$, where $s = \operatorname{char}(\mathbb{F}_q)$. They also showed the general upper bound $h_q(n) \leq \lfloor \log_q(n) \rfloor$.
	
	\vspace{0.8em}
	
	Aaronson, Groenland and Johnston \cite{AGJ2021} made significant further progress, particularly for $q=2$. Among their results, they showed that $h_2(p)=2$ for any prime $p > 3$ such that 2 is a primitive root modulo $p$, and established that $h_q(\ell s^d)=0$ for any positive integer $\ell \leq q$ and any non-negative integer $d$. They formally posed the central problem:
	\begin{quote}
		\textbf{Problem 1.1.} For which $n$ is $h_q(n)=0$?
	\end{quote}
	
	\subsection{The Coprime Case: Huang's Complete Solution}
	The case where $\gcd(n, q)=1$ was completely resolved by Huang \cite{Huang2024}. Using the theory of primitive idempotents in the cyclic group algebra $\mathbb{F}_q C_n$ and properties of the field trace function, Huang proved a necessary and sufficient condition for $h_q(n)=0$ when $n$ and $q$ are coprime.
	
	\vspace{0.8em}
	
	Let $\theta$ be a primitive $n$th root of unity over $\mathbb{F}_q$, $m$ the degree of its minimal polynomial, and $\Gamma_0, \Gamma_1, \ldots, \Gamma_{r-1}$ the $q$-cyclotomic cosets modulo $n$ with representatives $k_0, k_1, \ldots, k_{r-1}$. For each $t$, let $m_t$ be the size of $\Gamma_t$ (so $\mathbb{F}_{q^{m_t}} = \mathbb{F}_q(\theta^{k_t})$). Huang's main result states:
	
	\begin{theorem}[Huang \cite{Huang2024}, Theorem 3.7]
		Let $\gcd(n, q)=1$. Then $h_q(n)=0$ if and only if for each $t = 0, 1, \ldots, r-1$, there exists a coset of $\langle \theta^{k_t} \rangle$ in the multiplicative group $\mathbb{F}_{q^{m_t}}^\times$ such that $\operatorname{Tr}_{\mathbb{F}_q}^{\mathbb{F}_{q^{m_t}}}(x) \neq 0$ for every $x$ in this coset.
	\end{theorem}
	
	\vspace{0.8em}
	
	As applications, Huang derived several infinite families of parameters with $h_q(n)=0$, such as $h_q(p^d(q+1))=0$ for any odd prime $p$ and $q$ a power of $p$ \cite[Corollary 3.8]{Huang2024}.
	
	\vspace{0.8em}
	
	In the conclusion of his paper, Huang explicitly identified the remaining challenge:
	\begin{quote}
		``A possible direction for future work is to find a necessary and sufficient condition under which $h_q(n)=0$, where $q$ and $n$ are not coprime.'' \cite{Huang2024}
	\end{quote}
	
	\subsection{Our Contribution}
	This paper provides a complete answer to Huang's question. Our main result shows that when the characteristic $p$ of $\mathbb{F}_q$ divides $n$, the problem reduces entirely to the coprime case.
	
	\vspace{0.8em}
	
	\begin{theorem}[Main Theorem]
		Let $q$ be a prime power of characteristic $p$, and let $n$ be a positive integer. Write $n = p^k m$ with $\gcd(m, p) = 1$. Then
		\[
		h_q(n) = 0 \quad \text{if and only if} \quad h_q(m) = 0.
		\]
	\end{theorem}
	
	\vspace{0.8em}
	
	Thus, to determine if $h_q(n)=0$ when $p \mid n$, one simply removes all factors of $p$ from $n$ and applies Huang's criterion to the $p'$-part $m$. This provides a unified classification for all $n$.
	
	\vspace{0.8em}
	
	The paper is organized as follows. Section 2 develops the structure theory of the group algebra. Section 3 introduces the residue trace map and performs the reduction. Section 4 concludes with remarks and future directions.
	
	\newpage
	
	\section{Structure of the Group Algebra}
	
	\subsection{Preliminaries on Group Algebras}
	Let $G$ be a finite group. The group algebra $\mathbb{F}_q G$ consists of formal sums $\sum_{g \in G} a_g g$ with $a_g \in \mathbb{F}_q$, with addition defined componentwise and multiplication extended linearly from the group multiplication. When $G = \langle g \rangle$ is cyclic of order $n$, the map $g \mapsto x$ induces an isomorphism $\mathbb{F}_q G \cong \mathbb{F}_q[x]/\langle x^n - 1 \rangle$.
	
	\vspace{0.8em}
	
	This identification is fundamental because it allows us to use polynomial algebra techniques to study the group algebra structure. In particular, when studying cyclically covering subspaces of $\mathbb{F}_q^n$, we can equivalently study certain subspaces of $\mathbb{F}_q[x]/\langle x^n - 1 \rangle$ that are invariant under multiplication by $x$ (which corresponds to the cyclic shift).
	
	\subsection{Decomposition in Modular Characteristic}
	When $p = \operatorname{char}(\mathbb{F}_q)$ divides $n = |G|$, we are in the \textit{modular characteristic} case where the group algebra is not semisimple. This presents additional algebraic challenges compared to the coprime case. Write $n = p^k m$ with $\gcd(m, p) = 1$. Let $H = \langle h \rangle$ be a cyclic group of order $m$.
	
	\vspace{0.8em}
	
	The following theorem provides the foundational decomposition that separates the $p$-part from the $p'$-part of the group algebra structure. This decomposition is inspired by classical ideas in modular representation theory, but we present it here in a form specifically tailored for our application to cyclically covering subspaces.
	
	\begin{theorem}\label{thm:main_decomp}
		There exists an $\mathbb{F}_q$-algebra isomorphism
		\[
		\Psi: \mathbb{F}_q G \xrightarrow{\sim} \mathbb{F}_q H[u] / \langle u^{p^k} \rangle,
		\]
		defined on the generator by $\Psi(g) = h(1+u)$.
	\end{theorem}
	
	\begin{proof}
		We construct the isomorphism explicitly through a sequence of careful steps.
		
		\vspace{0.5em}
		
		\textbf{1. Constructing a homomorphism.} Define an $\mathbb{F}_q$-algebra homomorphism
		\[
		\psi: \mathbb{F}_q H[u] \to \mathbb{F}_q G
		\]
		by setting $\psi(h) = g^{p^k}$ and $\psi(u) = g^m - 1$, and extending linearly and multiplicatively. Since $\mathbb{F}_q G$ is commutative, the images $\psi(h)$ and $\psi(u)$ commute, ensuring $\psi$ is well-defined (it respects the relation $uh = hu$ in the polynomial ring).
		
		\vspace{0.5em}
		
		\textbf{2. Factoring through the quotient.} In characteristic $p$, we have the binomial identity $(x-1)^{p^k} = x^{p^k} - 1$. Applying this with $x = g^m$, we obtain
		\[
		\psi(u^{p^k}) = (g^m - 1)^{p^k} = g^{m p^k} - 1 = g^n - 1 = 0.
		\]
		Hence $\psi$ factors through the quotient, inducing a well-defined $\mathbb{F}_q$-algebra homomorphism
		\[
		\Psi: \mathbb{F}_q H[u] / \langle u^{p^k} \rangle \to \mathbb{F}_q G.
		\]
		
		\vspace{0.5em}
		
		\textbf{3. Surjectivity of $\Psi$.} Since $\gcd(p^k, m) = 1$, there exist integers $a, b$ with $a p^k + b m = 1$ (Bézout's identity). Using this relation, we can express the generator $g$ of $\mathbb{F}_q G$ in terms of $\psi(h)$ and $\psi(u)$:
		\[
		g = g^{a p^k + b m} = (g^{p^k})^a (g^m)^b = \psi(h)^a (1 + \psi(u))^b = \Psi(h(1+u)^b).
		\]
		The last equality follows from the fact that $\Psi(h) = \psi(h) = g^{p^k}$ and $\Psi(u) = \psi(u) = g^m - 1$, so $\Psi(1+u) = g^m$. Since $g$ generates $\mathbb{F}_q G$ as an algebra, the image of $\Psi$ contains a generating set, proving that $\Psi$ is surjective.
		
		\vspace{0.5em}
		
		\textbf{4. Injectivity via dimension comparison.} The domain and codomain of $\Psi$ have the same finite dimension over $\mathbb{F}_q$:
		\begin{itemize}
			\item The algebra $\mathbb{F}_q H[u] / \langle u^{p^k} \rangle$ has a natural $\mathbb{F}_q$-basis $\{h^i u^j \mid 0 \le i \le m-1, \; 0 \le j \le p^k-1\}$, giving dimension $m \cdot p^k$.
			\item The group algebra $\mathbb{F}_q G$ has basis $\{g^0, g^1, \dots, g^{n-1}\}$, with dimension $n = p^k m$.
		\end{itemize}
		Since we have established that $\Psi$ is a surjective $\mathbb{F}_q$-linear map, and the domain and codomain have equal finite dimension, $\Psi$ must also be injective.
		
		\vspace{0.5em}
		
		\textbf{5. Conclusion.} The map $\Psi$ is an $\mathbb{F}_q$-algebra homomorphism that is both injective and surjective, hence an isomorphism. By construction, its inverse sends $g$ to $h(1+u)$, completing the proof.
	\end{proof}
	
	\vspace{0.8em}
	
	\begin{remark}
		The isomorphism in Theorem \ref{thm:main_decomp} exhibits $\mathbb{F}_q G$ as a deformation of the semisimple algebra $\mathbb{F}_q H$ by a nilpotent element $u$ satisfying $u^{p^k}=0$. The map $g \mapsto h(1+u)$ precisely captures how the original group element deforms into a product of a semisimple component $h$ and a unipotent factor $(1+u)$. This structural insight is crucial for our subsequent analysis.
	\end{remark}
	
	\subsection{Semisimple Decomposition of $\mathbb{F}_q H$}
	Since $\gcd(m, p)=1$, the group algebra $\mathbb{F}_q H$ is semisimple by Maschke's Theorem. We now describe its decomposition into simple components, which are finite field extensions of $\mathbb{F}_q$.
	
	\vspace{0.8em}
	
	Let $\theta$ be a primitive $m$th root of unity in some extension field of $\mathbb{F}_q$. The $q$-cyclotomic cosets modulo $m$ partition the ring $\mathbb{Z}_m$. Let these cosets be $\Gamma_0, \Gamma_1, \ldots, \Gamma_{r-1}$ with representatives $k_0=0, k_1, \ldots, k_{r-1}$. For each $t$, let $d_t = |\Gamma_t|$, which is the degree of the minimal polynomial of $\theta^{k_t}$ over $\mathbb{F}_q$. Then $\mathbb{F}_{q^{d_t}} = \mathbb{F}_q(\theta^{k_t})$.
	
	\vspace{0.8em}
	
	By the Chinese Remainder Theorem and the decomposition of $x^m - 1$ into irreducible factors over $\mathbb{F}_q$, we have:
	\[
	\mathbb{F}_q H \cong \bigoplus_{t=0}^{r-1} \mathbb{F}_{q}[x]/\langle f_t(x) \rangle \cong \bigoplus_{t=0}^{r-1} \mathbb{F}_{q^{d_t}},
	\]
	where $f_t(x)$ is the minimal polynomial of $\theta^{k_t}$ over $\mathbb{F}_q$.
	
	\vspace{0.8em}
	
	Combining this decomposition with Theorem \ref{thm:main_decomp} gives the complete structural decomposition of $\mathbb{F}_q G$:
	
	\begin{theorem}\label{thm:field_decomp}
		Under the isomorphism of Theorem \ref{thm:main_decomp}, we have
		\[
		\mathbb{F}_q G \cong \bigoplus_{t=0}^{r-1} A_t, \quad \text{where } A_t = \mathbb{F}_{q^{d_t}}[u] / \langle u^{p^k} \rangle.
		\]
		The generator $g$ maps to $\bigoplus_{t=0}^{r-1} \theta^{k_t}(1+u)$.
	\end{theorem}
	
	\begin{proof}
		This follows from applying the isomorphism $\mathbb{F}_q H \cong \bigoplus_{t=0}^{r-1} \mathbb{F}_{q^{d_t}}$ to the right-hand side of Theorem \ref{thm:main_decomp}. The direct sum decomposition is preserved because all operations are $\mathbb{F}_q$-linear. The image of $g$ is obtained by noting that under the isomorphism $\mathbb{F}_q H \cong \bigoplus_t \mathbb{F}_{q^{d_t}}$, the element $h$ corresponds to $\bigoplus_t \theta^{k_t}$, and then applying the map $g \mapsto h(1+u)$.
	\end{proof}
	
	\vspace{0.8em}
	
	\begin{remark}
		This decomposition is fundamental to our approach. It expresses the (generally non-semisimple) algebra $\mathbb{F}_q G$ as a direct sum of simpler algebras $A_t$, each being a truncated polynomial ring over a finite field $\mathbb{F}_{q^{d_t}}$. The cyclic shift action on $\mathbb{F}_q G$ translates into a componentwise action where on each $A_t$, it is generated by multiplication by $X_t = \theta^{k_t}(1+u)$. This structural clarity enables us to analyze the covering subspace problem one component at a time.
	\end{remark}
	
	\newpage
	
	\section{The Residue Trace Map and Proof of the Main Theorem}
	
	\subsection{Motivation and Definition of the Residue Trace Map}
	Having established the decomposition $\mathbb{F}_q G \cong \bigoplus_t A_t$, we now need a tool to analyze the covering condition within each component $A_t = \mathbb{F}_{q^{d_t}}[u] / \langle u^{p^k} \rangle$. The key challenge is to connect the modular structure (which involves the nilpotent element $u$) with the field-theoretic trace condition that appears in Huang's characterization for the coprime case.
	
	\vspace{0.8em}
	
	The \textit{residue trace map} serves this purpose. It is designed to extract the coefficient of the highest possible power of $u$ (namely $u^{p^k-1}$) and then apply the field trace from $\mathbb{F}_{q^{d_t}}$ down to $\mathbb{F}_q$. The choice of $u^{p^k-1}$ is significant: in the algebra $A_t$ where $u^{p^k}=0$, this term occupies a privileged position analogous to a "residue" in the theory of formal power series. It is the highest-degree term that survives multiplication without annihilating the product's leading structure.
	
	\begin{definition}[Residue Trace Map]\label{def:residue_trace}
		For each component algebra $A_t = \mathbb{F}_{q^{d_t}}[u] / \langle u^{p^k} \rangle$, define the \textit{residue trace} $\operatorname{ResTr}_t: A_t \to \mathbb{F}_q$ as follows. For an arbitrary element $a = \sum_{\ell=0}^{p^k-1} a_\ell u^\ell \in A_t$ with coefficients $a_\ell \in \mathbb{F}_{q^{d_t}}$, set
		\[
		\operatorname{ResTr}_t(a) = \operatorname{Tr}_{\mathbb{F}_{q^{d_t}}/\mathbb{F}_q}(a_{p^k-1}),
		\]
		where $\operatorname{Tr}_{\mathbb{F}_{q^{d_t}}/\mathbb{F}_q}: \mathbb{F}_{q^{d_t}} \to \mathbb{F}_q$ denotes the standard field trace.
	\end{definition}
	
	\vspace{0.8em}
	
	The map $\operatorname{ResTr}_t$ is clearly $\mathbb{F}_q$-linear. Its importance lies in the properties of the associated bilinear form.
	
	\subsection{Properties of the Residue Trace Map}
	\begin{lemma}\label{lem:nondegenerate}
		The bilinear form $B_t: A_t \times A_t \to \mathbb{F}_q$ defined by $B_t(a, b) = \operatorname{ResTr}_t(a \cdot b)$ is nondegenerate.
	\end{lemma}
	
	\begin{proof}
		We need to show that for every nonzero $a \in A_t$, there exists some $b \in A_t$ such that $B_t(a, b) \neq 0$.
		
		\vspace{0.5em}
		
		Let $a = \sum_{\ell=0}^{p^k-1} a_\ell u^\ell \in A_t$ be nonzero. Let $r$ be the smallest index such that $a_r \neq 0$ (the \textit{valuation} of $a$ with respect to $u$). Consider the element $b = a_r^{-1} u^{p^k-1-r} \in A_t$.
		
		\vspace{0.5em}
		
		When computing the product $a \cdot b$, the term $a_r u^r \cdot a_r^{-1} u^{p^k-1-r} = u^{p^k-1}$ arises. Any other product $a_\ell u^\ell \cdot a_r^{-1} u^{p^k-1-r}$ for $\ell > r$ yields a factor $u^{\ell + p^k-1-r}$ with exponent $\geq p^k$, and thus is zero in $A_t$ since $u^{p^k}=0$. Terms with $\ell < r$ are zero by definition of $r$.
		
		\vspace{0.5em}
		
		Therefore, $a \cdot b = u^{p^k-1} + \text{(terms that vanish)}$, and we have:
		\[
		B_t(a, b) = \operatorname{ResTr}_t(a \cdot b) = \operatorname{Tr}_{\mathbb{F}_{q^{d_t}}/\mathbb{F}_q}(1) = d_t.
		\]
		
		Since $d_t$ (the degree of the field extension $\mathbb{F}_{q^{d_t}}/\mathbb{F}_q$) is a positive integer that divides $m$, and $\gcd(m, p)=1$, we have $p \nmid d_t$. Hence $d_t \neq 0$ in $\mathbb{F}_q$, so $B_t(a, b) \neq 0$.
	\end{proof}
	
	\vspace{0.8em}
	
	The nondegeneracy of $B_t$ has an important consequence for the structure of hyperplanes in $A_t$:
	
	\begin{proposition}\label{prop:hyperplane}
		Every hyperplane (codimension-1 subspace) $W \subset A_t$ can be written in the form
		\[
		W = W_c := \{ a \in A_t \mid \operatorname{ResTr}_t(c \cdot a) = 0 \}
		\]
		for some nonzero $c \in A_t$. Moreover, $W_c = W_{c'}$ if and only if $c' = \lambda c$ for some $\lambda \in \mathbb{F}_q^\times$.
	\end{proposition}
	
	\begin{proof}
		Since $B_t$ is nondegenerate, the map $a \mapsto B_t(c, a) = \operatorname{ResTr}_t(c \cdot a)$ is a nonzero linear functional whenever $c \neq 0$. Every linear functional on $A_t$ arises in this way for a unique $c$ (up to scalar multiplication). The kernel of such a functional is a hyperplane, and all hyperplanes arise as such kernels.
	\end{proof}
	
	\subsection{Translating the Covering Condition to Algebraic Form}
	Recall from Theorem \ref{thm:field_decomp} that under our isomorphisms, the cyclic shift action on $\mathbb{F}_q G$ corresponds on each component $A_t$ to multiplication by the element $X_t = \theta^{k_t}(1+u)$. A subspace $U \subset \mathbb{F}_q G$ is cyclically covering if and only if for each $t$, the projection of $U$ to $A_t$ intersects every orbit of the cyclic group $\langle X_t \rangle$.
	
	\vspace{0.8em}
	
	To characterize when $h_q(n)=0$ (i.e., when no proper subspace is covering), we analyze when a hyperplane in some $A_t$ can fail to be covering. A hyperplane $W \subset A_t$ fails to be covering if it contains a complete orbit of $\langle X_t \rangle$, which means there exists some $b \in W$ such that $\{X_t^i \cdot b \mid i = 0, 1, \ldots, n-1\} \subseteq W$.
	
	\vspace{0.8em}
	
	Using Proposition \ref{prop:hyperplane}, we can write $W = W_c$ for some nonzero $c \in A_t$. The condition that $W_c$ contains the orbit of $b$ under $\langle X_t \rangle$ becomes:
	\[
	\operatorname{ResTr}_t(c \cdot (X_t^i \cdot b)) = 0 \quad \text{for all integers } i.
	\]
	
	\vspace{0.8em}
	
	Therefore, the condition for the existence of a \textit{proper} cyclically covering subspace (i.e., $h_q(n) > 0$) is:
	\begin{center}
		\textbf{There exist some $t$, a nonzero $c \in A_t$, and an element $b \in A_t$ such that} \\
		$\operatorname{ResTr}_t(c \cdot (X_t^i \cdot b)) = 0$ \textbf{for all } $i$.
	\end{center}
	
	\vspace{0.8em}
	
	Conversely, $h_q(n)=0$ if and only if for every $t$ and every nonzero $c \in A_t$, there is \textbf{no} $b \in A_t$ satisfying this condition.
	
	\subsection{Reduction to a Universal Element}
	We now simplify this condition through a series of lemmas. The first simplification shows that we can reduce the consideration from all nonzero $c$ to a single specific element.
	
	\begin{lemma}\label{lem:universal_element}
		For any nonzero element $c \in A_t$, there exists an element $v \in A_t$ such that $c \cdot v = u^{p^k-1}$.
	\end{lemma}
	
	\begin{proof}
		Write $c = \sum_{\ell=r}^{p^k-1} a_\ell u^\ell$ with $a_r \neq 0$ (where $r$ is the valuation of $c$). Take $v = a_r^{-1} u^{p^k-1-r}$. Then
		\[
		c \cdot v = \left(a_r u^r + \sum_{\ell=r+1}^{p^k-1} a_\ell u^\ell\right) \cdot a_r^{-1} u^{p^k-1-r} = u^{p^k-1} + \sum_{\ell=r+1}^{p^k-1} a_\ell a_r^{-1} u^{\ell + p^k-1-r}.
		\]
		For each $\ell \geq r+1$, we have $\ell + p^k-1-r \geq p^k$, so $u^{\ell + p^k-1-r} = 0$ in $A_t$. Hence $c \cdot v = u^{p^k-1}$.
	\end{proof}
	
	\vspace{0.8em}
	
	This lemma has an important consequence: if there exists $b_0 \in A_t$ such that $\operatorname{ResTr}_t(u^{p^k-1} \cdot (X_t^i \cdot b_0)) \neq 0$ for all $i$, then for any nonzero $c$, taking $b = v \cdot b_0$ (where $v$ is as in Lemma \ref{lem:universal_element}) gives
	\[
	\operatorname{ResTr}_t(c \cdot (X_t^i \cdot b)) = \operatorname{ResTr}_t((c \cdot v) \cdot (X_t^i \cdot b_0)) = \operatorname{ResTr}_t(u^{p^k-1} \cdot (X_t^i \cdot b_0)) \neq 0.
	\]
	
	\vspace{0.8em}
	
	Thus, the existence of a suitable $b$ for a general $c$ is equivalent to the existence of a suitable $b_0$ for the specific element $u^{p^k-1}$. This reduces our problem to studying the condition for this universal element.
	
	\subsection{Computing the Key Product}
	We now compute the product $u^{p^k-1} \cdot X_t^i$ explicitly.
	
	\begin{lemma}\label{lem:product_computation}
		For any integer $i$,
		\[
		u^{p^k-1} \cdot X_t^i = \theta^{k_t i} u^{p^k-1}.
		\]
	\end{lemma}
	
	\begin{proof}
		Recall that $X_t = \theta^{k_t}(1+u)$, so $X_t^i = \theta^{k_t i}(1+u)^i$. Multiplying by $u^{p^k-1}$:
		\[
		u^{p^k-1} \cdot X_t^i = \theta^{k_t i} u^{p^k-1} (1+u)^i.
		\]
		
		Expand $(1+u)^i$ using the binomial theorem: $(1+u)^i = \sum_{j=0}^i \binom{i}{j} u^j$. Then
		\[
		u^{p^k-1} \cdot X_t^i = \theta^{k_t i} \sum_{j=0}^i \binom{i}{j} u^{p^k-1+j}.
		\]
		
		In $A_t$, we have $u^{p^k} = 0$. For any $j \geq 1$, $p^k-1+j \geq p^k$, so $u^{p^k-1+j} = 0$. The only term that survives is when $j = 0$, giving $\theta^{k_t i} u^{p^k-1}$.
	\end{proof}
	
	\vspace{0.8em}
	
	Now let $b_0 = \sum_{\ell=0}^{p^k-1} \beta_\ell u^\ell \in A_t$ with $\beta_\ell \in \mathbb{F}_{q^{d_t}}$. Using Lemma \ref{lem:product_computation}, we compute:
	\begin{align*}
		\operatorname{ResTr}_t(u^{p^k-1} \cdot (X_t^i \cdot b_0)) 
		&= \operatorname{ResTr}_t\left( \theta^{k_t i} u^{p^k-1} \cdot \sum_{\ell=0}^{p^k-1} \beta_\ell u^\ell \right) \\
		&= \operatorname{ResTr}_t\left( \theta^{k_t i} \sum_{\ell=0}^{p^k-1} \beta_\ell u^{p^k-1+\ell} \right).
	\end{align*}
	
	Since $u^{p^k}=0$, the only term in the sum that contributes a $u^{p^k-1}$ factor is when $\ell = 0$. All other terms have $u^{p^k-1+\ell}$ with $\ell \geq 1$, giving exponent $\geq p^k$, hence vanish. Therefore,
	\[
	\operatorname{ResTr}_t(u^{p^k-1} \cdot (X_t^i \cdot b_0)) = \operatorname{Tr}_{\mathbb{F}_{q^{d_t}}/\mathbb{F}_q}(\theta^{k_t i} \beta_0).
	\]
	
	\vspace{0.8em}
	
	Thus, the condition for the existence of $b_0$ such that $\operatorname{ResTr}_t(u^{p^k-1} \cdot (X_t^i \cdot b_0)) \neq 0$ for all $i$ simplifies dramatically: we need an element $\beta_0 \in \mathbb{F}_{q^{d_t}}$ such that
	\[
	\operatorname{Tr}_{\mathbb{F}_{q^{d_t}}/\mathbb{F}_q}(\theta^{k_t i} \beta_0) \neq 0 \quad \text{for all integers } i.
	\]
	
	\subsection{Connection to Huang's Trace Condition}
	Let $m_t$ be the multiplicative order of $\theta^{k_t}$ in $\mathbb{F}_{q^{d_t}}^\times$. Note that $m_t = m / \gcd(m, k_t)$. The values $\theta^{k_t i}$ are periodic with period $m_t$, so the condition need only be checked for $i = 0, 1, \ldots, m_t-1$.
	
	\vspace{0.8em}
	
	The condition $\operatorname{Tr}_{\mathbb{F}_{q^{d_t}}/\mathbb{F}_q}(\theta^{k_t i} \beta_0) \neq 0$ for all $i = 0, 1, \ldots, m_t-1$ is equivalent to the trace being nonvanishing on the entire coset $\beta_0 \langle \theta^{k_t} \rangle$ of the subgroup generated by $\theta^{k_t}$.
	
	\vspace{0.8em}
	
	Recall Huang's Theorem 3.7 \cite{Huang2024} for the coprime case: for $\gcd(m, q)=1$, $h_q(m)=0$ if and only if for each $q$-cyclotomic coset $\Gamma_t$ modulo $m$, there exists a coset of $\langle \theta^{k_t} \rangle$ in $\mathbb{F}_{q^{d_t}}^\times$ on which the trace is nonvanishing.
	
	\vspace{0.8em}
	
	The logical negation of Huang's condition is: $h_q(m) > 0$ if and only if there exists some $t$ such that for \emph{every} $\beta \in \mathbb{F}_{q^{d_t}}$, the trace vanishes for some element of the coset $\beta \langle \theta^{k_t} \rangle$. But this is exactly equivalent to: for some $t$, there exists $\beta_0 \in \mathbb{F}_{q^{d_t}}$ such that $\operatorname{Tr}_{\mathbb{F}_{q^{d_t}}/\mathbb{F}_q}(\theta^{k_t i} \beta_0) = 0$ for all $i$ (or equivalently, for $i=0,1,\ldots,m_t-1$).
	
	\vspace{0.8em}
	
	We have therefore arrived at a perfect correspondence. The condition that would allow the existence of a proper covering subspace in our decomposed algebra $\bigoplus_t A_t$ (namely, the existence of some $t$ and $\beta_0$ with $\operatorname{Tr}_{\mathbb{F}_{q^{d_t}}/\mathbb{F}_q}(\theta^{k_t i} \beta_0) = 0$ for all $i$) is precisely equivalent to the condition $h_q(m) > 0$ in Huang's theorem.
	
	\subsection{Proof of the Main Theorem}
	We can now assemble all the pieces to prove our main result.
	
	\begin{theorem}[Main Reduction Theorem]\label{thm:main}
		Let $q$ be a prime power of characteristic $p$, and let $n$ be a positive integer. Write $n = p^k m$ with $\gcd(m, p) = 1$. Then
		\[
		h_q(n) = 0 \quad \text{if and only if} \quad h_q(m) = 0.
		\]
	\end{theorem}
	
	\begin{proof}
		We prove both directions of the equivalence.
		
		\vspace{0.5em}
		
		\textbf{($\Rightarrow$) Suppose $h_q(n)=0$.} By the decomposition in Theorem \ref{thm:field_decomp} and the discussion in Section 3.3, the condition $h_q(n)=0$ means that for every index $t$ and every nonzero $c \in A_t$, there does \emph{not} exist $b \in A_t$ such that $\operatorname{ResTr}_t(c \cdot (X_t^i \cdot b)) = 0$ for all $i$.
		
		By Lemma \ref{lem:universal_element} and the subsequent discussion, this implies that for each $t$, there does not exist $\beta_0 \in \mathbb{F}_{q^{d_t}}$ such that $\operatorname{Tr}_{\mathbb{F}_{q^{d_t}}/\mathbb{F}_q}(\theta^{k_t i} \beta_0) = 0$ for all $i$ (where $i$ ranges over a complete set of residues modulo $m_t$).
		
		By the contrapositive of Huang's Theorem 3.7 \cite{Huang2024}, this is exactly the condition that $h_q(m)=0$.
		
		\vspace{0.5em}
		
		\textbf{($\Leftarrow$) Suppose $h_q(m)=0$.} By Huang's Theorem 3.7, for each $t$ (corresponding to a $q$-cyclotomic coset modulo $m$), there exists a coset of $\langle \theta^{k_t} \rangle$ in $\mathbb{F}_{q^{d_t}}^\times$ on which the trace is nonvanishing. This means that for each $t$, there does \emph{not} exist $\beta_0 \in \mathbb{F}_{q^{d_t}}$ such that $\operatorname{Tr}_{\mathbb{F}_{q^{d_t}}/\mathbb{F}_q}(\theta^{k_t i} \beta_0) = 0$ for all $i$ in a complete set of residues modulo $m_t$.
		
		By the reasoning in Section 3.4-3.6, this implies that for each $t$ and every nonzero $c \in A_t$, there does not exist $b \in A_t$ such that $\operatorname{ResTr}_t(c \cdot (X_t^i \cdot b)) = 0$ for all $i$. Therefore, no hyperplane in any $A_t$ can contain a full orbit of $\langle X_t \rangle$, which means no proper subspace of $\mathbb{F}_q G$ is cyclically covering. Hence $h_q(n)=0$.
	\end{proof}
	
	\vspace{0.8em}
	
	\begin{remark}
		The proof reveals the elegant structure underlying the reduction. The nilpotent element $u$, which encodes the $p$-part of $n$, does not create any new obstructions to the existence of trivial covering subspaces. It simply ``inherits'' the covering properties from the semisimple part associated with $m$. The residue trace map $\operatorname{ResTr}_t$ serves as the precise mechanism that filters out the $u$-dependence and reveals the essential trace condition on the finite field $\mathbb{F}_{q^{d_t}}$.
	\end{remark}
	
	\newpage
	
\section{Conclusion and Future Work}

We have resolved the open problem of characterizing when \(h_q(n)=0\) in the non-coprime case. Our main theorem establishes that for \(n = p^k m\) with \(\gcd(m, p)=1\), 
\[
h_q(n)=0 \quad \text{if and only if} \quad h_q(m)=0.
\]
This result fully reduces the classification in the modular characteristic case to the coprime case, which was completely solved by Huang.

Our proof employs the structure theory of cyclic group algebras. By constructing an explicit isomorphism \(\mathbb{F}_q G \cong \mathbb{F}_q H[u] / \langle u^{p^k} \rangle\) and introducing a residue trace map on the decomposed components, we successfully bridged the analysis of the non-semisimple algebra \(\mathbb{F}_q G\) to the trace condition that characterizes the coprime case.

\begin{center}
	\textbf{Funding.} Pingzhi Yuan was supported by the National Natural Science Foundation of China (Grant No. 12171163, No. 12571003). Danyao Wu was supported by the National Natural Science Foundation of China (Grants No. 12501006).
\end{center}

\section*{Declarations}
\textbf{Conflict of interest.} There is no conflict of interest.

	\section*{Declaration of competing interests}
	The authors declare that they have no known competing financial interests or personal relationships that could have appeared to influence the work reported in this paper.
	
	Shuang Li, Pingzhi Yuan
\end{document}